\documentclass[oneside,a4paper,11pt,notitlepage]{article}
\usepackage[left=0.8in, right=0.8in, top=1.5in, bottom=1.5in]{geometry}
\usepackage{pifont}
\usepackage{makecell}
\usepackage[T1]{fontenc} 
\usepackage[utf8]{inputenc} 
\usepackage[english]{babel}
\usepackage{lipsum} 
\usepackage{lmodern}
\usepackage{amssymb}
\usepackage{amsthm}
\usepackage{bm}
\usepackage{mathtools}
\usepackage{braket}
\usepackage{esint}
\newcommand{\abs}[1]{{\left|#1\right|}}
\newcommand{\norma}[1]{{\left\Vert#1\right\Vert}}

\usepackage{tabularx}
\usepackage{fontawesome}
\usepackage{booktabs}
\usepackage{graphicx}
\usepackage{tikz}
\usetikzlibrary{patterns}
\usepackage{multicol}
\usepackage{caption}
\usepackage{enumerate}
\usepackage{multicol}
\usepackage[skins,theorems]{tcolorbox}
\tcbset{highlight math style={enhanced,
		colframe=black,colback=white,arc=0pt,boxrule=1pt}}
\def\XXint#1#2#3{{\setbox0=\hbox{$#1{#2#3}{\int}$}
    \vcenter{\hbox{$#2#3$}}\kern-.5\wd0}}

\theoremstyle{definition}
\newtheorem{definizione}{Definition}[section]
\theoremstyle{plain}
\newtheorem{teorema}{Theorem}[section]

\newtheorem{lemma}[teorema]{Lemma}
\newtheorem{prop}[teorema]{Proposition}
\newtheorem{corollario}[teorema]{Corollary}
\theoremstyle{definition}
\newtheorem{esempio}{Example}[section]
\newtheorem{oss}[esempio]{Remark}

\DeclareMathOperator{\R}{\mathbb{R}}

\DeclareMathOperator{\diam}{\, \textup{diam}}

\makeatletter
\newcommand{\myfootnote}[2]{\begingroup
	\def\@makefnmark{}%
	\addtocounter{footnote}{-1}%
	\footnote{\textbf{#1} #2}
	\endgroup}
\makeatother

\usepackage{soul}

\definecolor{OliveGreen}{rgb}{0,0.6,0}

\usepackage{hyperref}
\hypersetup{linktoc=none, bookmarksnumbered, colorlinks=true, linkcolor=red}
 \title{Geometrical bounds for the torsion and the first eigenvalue of the Laplacian with Robin boundary condition}
\author{Rosa Barbato, Alba Lia Masiello, Rossano Sannipoli}
\date{}

\newcommand{\Addresses}{{
  \bigskip 
   \textit{E-mail address}, R.~Barbato: \texttt{rosa.barbato2@unina.it} 
   
 \medskip
  \textit{E-mail address}, A.L.~Masiello: \texttt{albalia.masiello@unina.it} 
   
 \medskip
  \textsc{Dipartimento di Matematica e Applicazioni ``R. Caccioppoli'', Universit\`a degli studi di Napoli Federico II, Via Cintia, Complesso Universitario Monte S. Angelo, 80126 Napoli, Italy.}
  \medskip
 
  \textit{E-mail address}, R.~Sannipoli: \texttt{rossano.sannipoli@fjfi.cvut.cz} 
  
     \medskip 
\textsc{Department of Mathematics, Faculty of Nuclear Sciences and Physical Engineering, Czech Technical University in Prague, Trojanova 13, 120 00, Prague, Czech Republic.}\par\nopagebreak 

}} 
\makeatletter
\def\Cline#1#2{\@Cline#1#2\@nil}
\def\@Cline#1-#2#3\@nil{%
  \omit
  \@multicnt#1%
  \advance\@multispan\m@ne
  \ifnum\@multicnt=\@ne\@firstofone{&\omit}\fi
  \@multicnt#2%
  \advance\@multicnt-#1%
  \advance\@multispan\@ne
  \leaders\hrule\@height#3\hfill
  \cr}
\makeatother

\definecolor{verde}{RGB}{20,150,100}
\definecolor{purple}{RGB}{200,30,200}

\begin{document}
\maketitle
\begin{abstract}
In this paper, we deal with functionals involving the torsion and the first eigenvalue of the Laplacian with Robin boundary conditions (to which we refer as Robin Torsion and Robin Eigenvalue), with other geometrical quantities, in the class of convex sets. Firstly, we prove an upper bound for the Robin Torsion in terms of the $L^1$ and $L^2$ norms of the distance function from the boundary, which allows us to prove a generalization of the  Makai inequality involving the Robin Torsion, the Lebeasgue measure, and the inradius of a convex set. Subsequently, we prove quantitative estimates for the Robin Makai functional and for the Robin P\'olya functionals, which link the Lebesgue measure and the perimeter with the Robin Torsion and the Robin Eigenvalue respectively. In particular, we prove that the optimal values of all these shape functionals are achieved by slab domains.
\end{abstract}
\begin{center}
\begin{minipage}{11cm}
\small
\tableofcontents
\end{minipage}
\end{center}
\section{Introduction}
Let $\Omega\subset\mathbb R^n$ be an open, bounded set with Lipschitz boundary. Let $\beta>0$ be a positive parameter and let us consider the following problems
\begin{equation}\label{eq:RobinTorsionEigenvalue}
\textbf{(TP)}\;\begin{cases}
\displaystyle-\Delta u = 1 & \text{in}\;\Omega\\
\displaystyle\frac{\partial u}{\partial \nu}+\beta u = 0 & \text{on}\;\partial\Omega, 
\end{cases}\qquad \qquad \qquad \textbf{(EP)}\;\begin{cases}
\displaystyle-\Delta u = \lambda u& \text{in}\;\Omega\\
\displaystyle\frac{\partial u}{\partial \nu}+\beta u = 0 & \text{on}\;\partial\Omega, 
\end{cases}
\end{equation}
which are known in literature as the Torsion problem and the Eigenvalue problem with Robin boundary conditions, respectively. In particular, the $L^1$-norm of the unique solution to problem \textbf{(TP)} is called Robin Torsion of the set $\Omega$, which has also the following variational characterization:
\begin{equation}\label{torsion:robin}
    T_\beta(\Omega) = \sup_{v\in H^1(\Omega)}\frac{\displaystyle\bigg(\int_\Omega v\,dx\bigg)^2}{\displaystyle \int_\Omega\abs{\nabla v}^2\,dx+\beta\int_{\partial\Omega}v^2\,d\mathcal{H}^{n-1}}.
\end{equation}
The spectrum of the Laplacian for the problem \textbf{(EP)} is discrete, and composed by a sequence of positive and positively diverging eigenvalues. The first eigenvalue of this sequence is called first Robin eigenvalue and has the following variational characterization
\begin{equation}\label{eq:varcareig}
    \lambda_\beta(\Omega)= \inf_{v\in H^1(\Omega)}\frac{\displaystyle \int_\Omega\abs{\nabla v}^2\,dx+\beta\int_{\partial\Omega}v^2\,d\mathcal{H}^{n-1}}{\displaystyle\int_\Omega v^2\,dx}.
\end{equation}
In the present paper, we establish qualitative and quantitative relations in the class of open, bounded convex sets  for shape functionals that involve $T_\beta(\Omega)$ and $\lambda_\beta(\Omega)$ with  the Lebesgue measure, the perimeter and the inradius of  $\Omega$. The interest in such topic dates back to \cite{makai, polya1960}, where geometric inequalities for the Dirichlet Torsion and the Dirichlet eigenvalue are studied. We start by listing some results that will help clarify the framework and the quantities involved.
\subsection*{Some state of art on the Dirichlet case} Let us focus, for the moment, in the case of Dirichlet boundary conditions, which corresponds  to $\beta=+\infty$. Let us denote by $T(\Omega)$ and $\lambda(\Omega)$ the Dirichlet Torsion and the first Dirichlet Eigenvalue, respectively. First of all, let us stress that $T(\Omega)$ is monotonically increasing and $\lambda(\Omega)$ is monotonically decreasing with respect to the set inclusion,  and they satisfy the following scaling properties: 
\begin{equation*}
     T(t\Omega)= t^{n+2}T(\Omega),\qquad \lambda(t\Omega)= t^{-2}\lambda(\Omega), \qquad \qquad \forall t>0.
\end{equation*}
These two quantities, seen as functionals of $\Omega$, have attracted the attention of many mathematicians in the last century due to their physical and engineering applications. For instance, in the Dirichlet case and for $n=2$, the Torsion of $\Omega$ represents the torsional rigidity of a three-dimensional bar with constant cross-section $\Omega$, which is the resistance of the bar to be bended. The first eigenvalue, instead, can be interpreted as the fundamental frequency of a bi-dimensional vibrating membrane whose shape is $\Omega$ and which is fixed to the boundary. Finding the best shapes that optimize these functionals under some geometrical constraints is the goal of shape optimization problems. Two classical inequalities identify the ball as the optimal set under a volume constraint, in the class of open and bounded sets. Let $\Omega \subset \mathbb{R}^n$ be an open set with finite Lebesgue measure, and let $B$ denote a ball. The first result is the well-known Saint-Venant inequality, originally conjectured in \cite{stvenant}, which can be written in the scale-invariant form
\begin{equation*}
    \abs{\Omega}^{-\frac{n+2}{n}}T(\Omega) \le \abs{B}^{-\frac{n+2}{n}}T(B),
\end{equation*}
where $|\Omega|$ stands for the Lebesgue measure of $\Omega$. The second is the Faber–Krahn inequality, stating that
\begin{equation*}
    \abs{\Omega}^\frac{2}{n}\lambda(\Omega) \ge \abs{B}^\frac{2}{n}\lambda(B). 
\end{equation*}
Moreover, since the latter half of the twentieth century, several further inequalities connecting $T(\Omega)$ and $\lambda(\Omega)$ have been explored (see, for example, \cite{KJ2,KJ1,PS}). \\
Bounds of the torsion and the first eigenvalue have been investigated in terms of other geometrical quantities, in the class of convex sets. The pioneers in this direction have been  P\'olya, Makai and Hersch. Denoting  by $P(\Omega)$ and $r(\Omega)$ the perimeter and  the inradius of $\Omega$,  respectively (see Section \ref{sec2:preliminaries} for the precise definitions), it holds that

\begin{equation}\label{eq:dirichletfunctionals}
\begin{aligned}
\textbf{(I)}\;\;\frac{1}{3} \leq  \frac{T(\Omega) P^2(\Omega)}{\abs{\Omega}^3} \leq \frac{2}{3}\qquad\qquad\qquad & \textbf{(II)}\;\;\frac{\pi^2}{4n^2}\le\frac{\lambda( \Omega) \abs{\Omega}^2}{P^2(\Omega)}\le \frac{\pi^2}{4} \\
\\[-0.4em]
\textbf{(III)}\;\;\frac{1}{n(n+2)} \le \frac{T(\Omega)}{r(\Omega)^2\abs{\Omega}} \leq \frac{1}{3} \qquad\qquad\qquad & \textbf{(IV)}\;\;\frac{\pi^2}{4}\le\lambda( \Omega)r(\Omega)^2\le \lambda_1(B_1).
\end{aligned}
\end{equation}
We here briefly discuss the literature behind these bounds. The bounds in \textbf{(I)} were proved, in the planar setting, by Makai and  P\'olya in \cite{makai, polya1960}. Moreover they proved that these estimates cannot be improved. In particular, the lower bound is asymptotically achieved by a sequence of thinning rectangles, while the upper one by a sequence of thinning triangles. The left-hand inequality in \textbf{(I)} was later extended to any dimension in \cite{gavitone_2014}, where it was established for all open, bounded, convex subsets of $\mathbb{R}^n$, again with optimality achieved by a family of thinning cylinders. Additional generalizations are presented in \cite{AMPS,BBP2021}. Regarding \textbf{(II)}, 
the upper estimate was first established in \cite{polya1960} in the planar case, with sharpness attained by sequences of thinning rectangles. This result was subsequently extended in \cite{gavitone_2014}  to higher dimensions and, more generally, to the first eigenvalue of the anisotropic $p$-Laplacian. The lower bound was obtained by Makai in dimension two \cite{makai}, and later generalized to all dimensions in \cite{brasco} and to the anisotropic case in \cite{DBN}. Turning to \textbf{(III)},
Makai proved the upper estimate in two dimensions (see \cite{makai}), also showing that it is sharp, with extremal behavior exhibited by a sequences of thinning rectangles. The lower estimate originates in \cite{PS}, where equality holds only for disks. Later, both bounds were extended to arbitrary dimensions and to more general operators in \cite{dellapietra_gavitone2018}, where it is further shown that the upper bound is achieved by a suitable family of thinning cylinders. Finally, concerning the functional in \textbf{(IV)}, the right-hand inequality follows immediately from monotonicity under set inclusion. The left-hand inequality is the classical Hersch–Protter inequality: it was first proved by Hersch in two dimensions \cite{hersch1960frequence} and subsequently extended to all dimensions by Protter \cite{protter1981lower}. Additional developments can be found in \cite{brasco_2020_principal_frequencies,brasco2018principal, prinari2023sharp, MS}.
\\

\noindent 
Only recently, the lower bound in \textbf{(I)} and \textbf{(IV)}, and the upper bound in \textbf{(II)} and \textbf{(III)}, have been improved, by proving continuity inequalities and quantitative versions in the papers \cite{AGS, AMPS}. Here, the authors managed to add geometrical remainder terms which allowed to better understand the nature of the optimal sequences. 

\subsection*{Back to the Robin case}
One of our aim is to extend this analysis to the Robin case, so let now be $\beta \in (0,+\infty)$. The presence of Robin boundary conditions drastically changes  the problems studied. The first big difference with the Dirichlet case is that there is no monotonicity with respect to the set inclusion. Moreover, the scaling properties do not hold in a standard way; indeed, it holds 
\begin{equation*}
    T_{\beta/t}(t\Omega)=t^{n+2}T_\beta(\Omega), \qquad \qquad \lambda_{\beta/t}(t\Omega)=t^{-2}\lambda_\beta(\Omega), \qquad \qquad \forall t>0.
\end{equation*}
Regarding the shape optimization issue, it is known that the Robin counterpart of the Saint-Venant and Faber-Krahn inequalities holds. Concerning the Robin Torsion, the authors in \cite{BG2015}, proved in the class of bounded Lipschitz sets in $\mathbb R^n$, that
\begin{equation*}
    T_\beta(\Omega)\le T_\beta(\Omega^\sharp)
\end{equation*}
where $\Omega^\sharp$ is the ball centered at the origin having the same measure as $\Omega$. Although the result is analogous to the Dirichlet case, the techniques used to prove it are completely different. This is due to the fact that the solution to \textbf{(TP)} is not constant in general on $\partial \Omega$ and therefore all rearrangements arguments or Talenti inequalities fail.\\
About the first Robin eigenvalue, the corresponding Faber-Krahn inequality has been proved by Bossel \cite{Bos1986} in dimension $2$ and subsequently by Daners \cite{Dan2006} in higher dimensions. Nowadays, this inequality is known as the Bossel-Daners inequality and guarantees that in the class of bounded Lipschitz sets in $\mathbb{R}^n$ we have
\begin{equation*}
    \lambda_\beta(\Omega)\ge \lambda_\beta(\Omega^\sharp),
\end{equation*}
where again $\Omega^\sharp$ is the ball centered at the origin having the same measure as $\Omega$. Surprisingly, a rearrangement technique in this case is useful to prove the above inequality.\\

As already mentioned, one of the aims of this paper is to obtain estimates for $\lambda_\beta(\Omega)$ and $T_\beta(\Omega)$ in terms of geometric quantities related to $\Omega$, in some case improving in a quantitative way the results obtained in \cite{DP2024, DPP, Savo, sperb}. \\ 
For a smooth, bounded, and  mean convex domain $\Omega$, in \cite{DPP} the authors prove a lower bound for the first Robin eigenvalue in terms of the first eigenvalue of a one-dimensional problem, and, as a consequence, a lower estimate in terms of the inradius $r(\Omega)$ of the set, 

\begin{equation}
\lambda_{\beta}(\Omega)\ge 
\left( \frac{\pi}{2}\right)^2\frac{1}{\left(r(\Omega)+\frac{\pi}{2}\beta^{-1}\right)^2}.
\end{equation}
This inequality, for $\beta\to +\infty$, reduces to the Hersh-Protter inequality, which is the lower bound in \textbf{(IV)}. In \cite{DP2024}, the author studies the problem of bounding the Robin Torsion and the Robin Eigenvalue in terms of perimeter and measure of the set $\Omega$. The proof of these results rely on a dimensional reduction argument and on  the study of the one-dimensional problem for the Robin eigenvalue and the Robin Torsion. The one-dimensional eigenvalue problem taken into account is the following 
\begin{equation}
\label{1dim_intro}
    \begin{cases}
    X''+\nu X=0 \quad\text{in}\ (0,s_0),\\
    X'(0)=0, \\
    X'(s_{0})+\beta X(s_{0})=0. 
    \end{cases}
\end{equation} 
If we denote by $\nu_1(\beta,s_0)$ the first eigenvalue of \eqref{1dim_intro} and we choose $s_0=\abs{\Omega}/P(\Omega)$, then the author proves that among open and bounded convex sets in $\R^n$ it holds

\begin{equation}
    \label{dp:eigen}
    \lambda_\beta(\Omega)\le \nu_1\left(\beta, \frac{\abs{\Omega}}{P(\Omega)}\right),
\end{equation}
being this inequality asymptotically sharp on slabs, which are the cartesian product between an interval and $\mathbb R^{n-1}$ (see Section \ref{sec2:preliminaries} for the precise definition). In this case, the dependence on the geometry of the upper bound is not explicit as in $\textbf{(II)}$, but it is possible, by exploiting the property of $\nu_1(\beta, \cdot)$ (see subsection \ref{1D_subsec}),  to prove an explicit bound

\begin{equation}\label{explicit}
    \lambda_\beta(\Omega)\le \frac{\pi^2}{4}\left(\frac{1}{1+\frac{2}{ r(\Omega)\beta}}\right) \frac{P^2(\Omega)}{\abs{\Omega}^2}.
\end{equation}
The price to pay in order to have an explicit geometrical upper bound is the loss of the sharpness of the estimate. Nevertheless,
the latter inequality is the generalization to the Robin setting of the upper bound in \textbf{(II)}, and it reduces to it when $\beta$ diverges.\\ 
Lastly, for the Robin Torsion, a P\'olya-type lower bound was proved in \cite{DP2024}

\begin{equation}
    \label{lower:torsion:robin}
    \dfrac{T_\beta(\Omega)P^2(\Omega)}{|\Omega|^3}\ge\bigg(\dfrac{1}{3}+\frac{1}{r(\Omega)\beta}\bigg),
\end{equation}
and, once again, one can see that for $\beta$ large we recover the lower bound in \textbf{(I)}, but in this case, the minimizing sequences were not characterized.\\

Our first aim is to obtain bounds for the Robin Torsion in terms of the measure and the inradius of the set, obtaining a Robin counterpart of $\textbf{(III)}$. In order to do that, we need to bound the Robin Torsion in terms of the $L^1$ and $L^2$-norms of the distance function from the boundary of $\Omega$ (see Section \ref{sec2:preliminaries} for the definition and properties). In the case of the Dirichlet Torsion, it was proved by Makai (see \cite{makai}) in two dimensions and in \cite{prinari2023sharp} in higher dimensions and in a more general setting. The result is the following 
\begin{teorema}\label{thm:TlessD}
    Let $\Omega$ be a bounded, open, nonempty and convex set of $\mathbb{R}^n$. Then
    \begin{equation}\label{eq:torsiondistancefunction}
        T_\beta(\Omega)\le \int_{\Omega}d(x,\partial \Omega)^2\,dx+\frac{2}{\beta}\int_{\Omega}d(x,\partial \Omega)\,dx.
    \end{equation}
\end{teorema}
We highlight that to prove Theorem \ref{thm:TlessD} we use the demonstration technique implemented in \cite{makai}. 
As already pointed out before, this estimate allows us to prove a Makai-type inequality in the Robin case and it is contained in the next corollary.

\begin{corollario}
    \label{thm:geometricmakai}
Let $\Omega$ be a bounded, open, nonempty and convex set of $\mathbb{R}^n$. Then
\begin{equation}\label{eq:qualitativemakai}
    \frac{T_\beta(\Omega)}{r(\Omega)^2\abs{\Omega}}\le \frac{1}{3}+\frac{1}{r(\Omega)\beta}.
\end{equation}
\end{corollario}

We stress that we did not mention anything about the sharpness of inequalities \eqref{lower:torsion:robin} and \eqref{eq:qualitativemakai}, and the reason is that we need one inequality to prove the sharpness of the other.\\
Hence, our second aim is to study the equality case of inequalities \eqref{lower:torsion:robin} and \eqref{eq:qualitativemakai} and, in order to do that, we actually focus on the stability issues.


All the P\'olya-type inequalities are asymptotically sharp along a sequence of slabs. The lack of an optimal set in these inequalities might appear to be an obstacle when studying the corresponding stability problem. To overcome this issue, in \cite{ABF,ABF2, AGS, AMPS} two new remainder terms were introduced. 
The one we will focus on is denoted by $\mathcal{R}(\Omega)$, and it is a quantity between $0$ and $n$, defined as
$$\mathcal{R}(\Omega):=\frac{P(\Omega)r(\Omega)}{\abs{\Omega}}-1.$$
The infimum $0$ is achieved, for instance, by sequences of thinning cylinders or slabs (see Section $2$, Proposition \ref{prop:PRM-1}) and, for this reason, the remainder term $\mathcal{R}(\Omega)$ will play a central role in the next quantitative results. The first one we introduce is a quantitative result about the Makai Robin inequality, and it is the following

\begin{teorema}\label{thm:quantitativemakai}
    Let $\Omega\subset \mathbb R^n$ be an open, bounded, nonempty and convex. Then it holds
    \begin{equation}\label{eq:MakaiRobin}
    2\left(\frac{1}{3}+\frac{1}{r(\Omega)\beta}\right)\mathcal{R}(\Omega)\ge\left(\dfrac{1}{3}+\dfrac{1}{\beta r(\Omega)}\right)-
        \dfrac{T_{\beta}(\Omega)}{\abs{\Omega}r(\Omega)^2}\ge C_1\mathcal{R}(\Omega),
    \end{equation} with $C_1=\dfrac{n+1}{3n(2n-1)}$. As a consequence, sequence of slab-like domain are sharp for inequality \eqref{eq:qualitativemakai}.
\end{teorema}

Regarding the P\'olya inequality for the  Robin Torsion, we get
\begin{teorema}\label{thm:quantitativepolyatorsion}
    Let $\Omega\subset\R^n$ be an open, bounded., nonempty and convex set. Then, it holds

    \begin{equation}
        \label{polya:lower:R}
        (n+1)\bigg(\dfrac{1}{3}+\frac{1}{r(\Omega)\beta}\bigg)\mathcal{R}(\Omega)\ge \frac{T_\beta(\Omega) P^2(\Omega)}{\abs{\Omega}^3}- \left(\frac13+ \frac{1}{\beta r(\Omega)}\right)\ge C_2 \mathcal{R}^3(\Omega),
    \end{equation}
 with $C_2=\dfrac{1}{2^3\cdot3^4  n^3}.$
     As a consequence, sequence of slab-like domain are sharp for inequality \eqref{lower:torsion:robin}.
\end{teorema}
The last main Theorem is the quantitative version of the inequality \eqref{dp:eigen}.
\begin{teorema}\label{thm:eigen_R}
       Let $\Omega\subset\R^n$ be an open, bounded, nonempty and convex set. Let $s_0=\dfrac{\abs{\Omega}}{P(\Omega)}$, and let $\nu_{1}(\beta,s_0)$ be the eigenvalue defined in \eqref{def_eigenvalue_mu}. Then, it holds

    \begin{equation}\label{stima:autoval:R}
       K_1\mathcal{R}(\Omega) \geq \frac{\nu_{1}(\beta,s_0)\abs{\Omega}^2}{P^2(\Omega)} -\frac{\lambda_\beta(\Omega)\abs{\Omega}^2}{P^2(\Omega)}\ge K_2\mathcal{R}(\Omega)^4,
    \end{equation}
where
\begin{equation*}
    K_1= \frac{\pi^2}{2}\sqrt{1+\frac{4\beta^2}{\pi^2}(r(\Omega)^2+\frac{\pi^2}{4\beta}r(\Omega))}, \qquad\qquad K_2 = \dfrac{1}{2\cdot 3^4 \pi (2n-1)n^3}\left(\frac{\pi^2}{4}\frac{1}{1+\frac{\pi^2 n}{4\beta r(\Omega)}}\right)^4.
\end{equation*}
\end{teorema}

As a corollary of Theorem \ref{thm:eigen_R} and the bound \eqref{explicit}, we can obtain the following

\begin{corollario}\label{thm:quantitativepolyaeigenvalue}
    Let $\Omega\subset\R^n$ be an open, bounded and convex set. Then, it holds

    \begin{equation}
        \label{eigen:polya:thm}
        \frac{\pi^2}{4}\left(\frac{1}{1+\frac{2}{ r(\Omega)\beta}}\right) -\frac{\lambda_\beta(\Omega)\abs{\Omega}^2}{P^2(\Omega)}\ge C_3\left(\frac{\pi^2}{4}\frac{1}{1+\frac{\pi^2 n}{4\beta r(\Omega)}}\right)^2\mathcal{R}(\Omega)^4,
    \end{equation}
with $C_3=\dfrac{1}{2\cdot 3^4 \pi (2n-1)n^3}.$
\end{corollario}


Finally, let us note that another remainder term was introduced in \cite{AGS}, which we denote by $\mathcal{A}(\Omega)$, and it is a measure of the "thinness" of the set

$$\mathcal{A}(\Omega)= \frac{w_\Omega}{d(\Omega)},$$
  where $w_\Omega$ and $d(\Omega)$ represent the minimal width and diameter of the set (for the definition see Section \ref{sec2:preliminaries}).  As proved in \cite{AGS}, there exists a positive dimensional constant $K=K(n)>0$, such that $\mathcal{R}(\Omega)\ge K\mathcal{A}(\Omega)$, and the reverse inequality is not true. For this reason, the lower bound in \eqref{eq:MakaiRobin}, \eqref{polya:lower:R} and \eqref{eigen:polya:thm} can be rewritten also in terms of $\mathcal{A}(\Omega)$, where the sharp exponent is ensured only in equality \eqref{eq:MakaiRobin}.
For completeness, in Section \ref{section_alpha}, we manage to prove that the exponent is sharp even for the other two inequalities, as proved in Proposition \ref{prop:quantitative_wd} and in Proposition \ref{prop:quantitative_wd2}.\\

\noindent \textbf{Plan of the paper:} In Section $2$ we recall some basic notions and definitions, and we recall some classical results, focusing in particular on the class of convex sets. In Section $3$ we prove the upper bounds for the Robin Torsion, that are Theorem \ref{thm:TlessD} and Corollary \ref{thm:geometricmakai}. In Section $4$ and $5$ we prove the quantitative estimates involving $\mathcal{R}(\Omega)$, which are the content of Theorems \ref{thm:quantitativemakai}-\ref{thm:quantitativepolyatorsion}  and \ref{thm:eigen_R},  giving also some remark about the nature of the optimal sequences. Lastly, in section $6$ we prove quantitative estimates involving the remainder term $\mathcal{A}(\Omega)$. 

\section{Preliminaries}\label{sec2:preliminaries}
\subsection{Notations and basic facts}
 Throughout this article, $|\cdot|$ will denote the Euclidean norm in $\mathbb{R}^n$,
 while $\cdot$ is the standard Euclidean scalar product for  $n\geq2$. By $\mathcal{H}^k(\cdot)$, for $k\in [0,n)$, we denote the $k-$dimensional Hausdorff measure in $\mathbb{R}^n$. The perimeter of $\Omega$ in $\mathbb{R}^n$ will be denoted by $P(\Omega)$ and, if $P(\Omega)<\infty$, we say that $\Omega$ is a set of finite perimeter. In our case, $\Omega$ is a bounded, open and convex set; this ensures us that $\Omega$ is a set of finite perimeter and that $P(\Omega)=\mathcal{H}^{n-1}(\partial\Omega)$. Some references for results relative to the sets of finite perimeter and for the coarea formula are, for instance, \cite{ambrosio2000functions,maggi2012sets}.

 We give now the minimal width (or thickness) of a convex set.
\begin{definizione}\label{support}
  Let $\Omega$ be a bounded, open and convex set of $\mathbb{R}^n$. The support function of $\Omega$ is defined as
  \begin{equation*}
    h_\Omega(y)=\sup_{x\in \Omega}\left(x\cdot y\right), \qquad y\in \mathbb{R}^n .
  \end{equation*}

The width of $\Omega$ in the direction $y \in \mathbb{R}$ is defined as 
  \begin{equation*}
    \omega_{\Omega}(y)=h_{\Omega}(y)+h_{\Omega}(-y)
    \end{equation*}
 and the minimal width of $\Omega$ as
\begin{equation*}
    w_\Omega=\min\{  \omega_{\Omega}(y)\,|\; y\in\mathbb{S}^{n-1}\}.
\end{equation*}
\end{definizione}
We will denote by $r(\Omega)$ the inradius of $\Omega$, i.e.
 \begin{equation}
     \label{inradiuss}
     r(\Omega)=\sup\{r\in \mathbb R: B_r(x)\subset \Omega, x\in \Omega\},
 \end{equation}
and by $\diam(\Omega)$ the diameter of $\Omega$, that is
\begin{equation*}
    \diam(\Omega) = \sup_{x,y\in\Omega}|x-y|.
\end{equation*}
\begin{definizione}
    Let $\Omega_\ell$ be a sequence of non-empty, bounded, open and convex sets of $\mathbb{R}^n$. $\Omega_\ell$ is a sequence of  thinning domains if
    \begin{equation*}
        \dfrac{w_{\Omega_\ell}}{\diam(\Omega_\ell)}\xrightarrow{\ell \to 0}0.
    \end{equation*}

\end{definizione}
A particular sequence of thinning domains are what we call a slab-like domain.
\begin{definizione} 
Let $K\subset \R^{n-1}$ be a bounded convex set, we define a sequence of \emph{slab-like domains} as
$$\Omega_\ell=(-a_\ell,a_\ell)\times \frac{1}{\ell}K, \qquad \; \ell \to 0,$$
   where $a_\ell$ is a positive converging sequence. 

\noindent  On the other hand, we define the sequence of \emph{thinning cylinders} as sequence 
    \begin{equation}\label{thin_rect}
    C_\ell=\left(-\frac{\ell}{2}, \frac{\ell}{2}\right) \times K.
\end{equation}
\end{definizione}

We recall in the following Proposition the relation between the inradius and the minimal width (see for example \cite{ santalo_sobre, scott_family, scott_convex_2000,fenchel_bonnesen}). 
\begin{prop}
\label{prop:r>w}
Let $\Omega$ be a bounded, open, and convex set of $\mathbb{R}^n$. Then, the following estimates hold:
\begin{equation}
\label{eq:lowboundinradius}
    \displaystyle \frac{w_\Omega}{2}\ge r(\Omega)\ge\begin{cases}
    w_{\Omega} \displaystyle{\frac{\sqrt{n+2}}{2n+2}} & n \,\, \text{even}\\ \\
    w_{\Omega} \displaystyle{\frac{1}{2\sqrt{n}}} & n \,\, \text{odd},
    \end{cases}
\end{equation}
\end{prop}

\noindent Moreover, the following estimate involving the perimeter and the diameter holds
\begin{equation}\label{eq:perdiam}
\displaystyle{P(\Omega)\le n\omega_n \left(\frac{\diam(\Omega)}{2}\right)^{n-1} }.
 \end{equation}

\subsection{Distance to the boundary and inner parallel sets}\label{innere}

Let $\Omega\subset\mathbb{R}^n$ be a nonempty, bounded, open and convex set.  
We introduce the distance from the boundary 
\[
d(x,\partial\Omega)=\inf_{y\in\partial\Omega}|x-y|,\qquad x\in\Omega.
\]
As a consequence of the convexity of $\Omega$, the function $d(\cdot,\partial\Omega)$ is concave in $\Omega$.

For every $t\in[0,r(\Omega)]$, where $r(\Omega)$ denotes the inradius of $\Omega$, we define the associated inner parallel sets by
\[
\Omega_t:=\{x\in\Omega:\ d(x,\partial\Omega)>t\}.
\]
We use the notation
\[
\mu(t):=|\Omega_t|,\qquad P(t):=P(\Omega_t),\qquad t\in[0,r(\Omega)].
\]
Being $|\nabla d|=1$ a.e. in $\Omega$, the Coarea formula gives
\[
\mu(t)=\int_{\{d>t\}}dx
      =\int_t^{r(\Omega)}\!\!\int_{\{d=s\}} d\mathcal{H}^{n-1}\,ds
      =\int_t^{r(\Omega)} P(s)\,ds.
\]
As a consequence, the function $\mu$ is absolutely continuous and decreasing in $[0,r(\Omega)]$, and satisfies
\begin{equation}\label{eq:dermu}
\mu'(t)=-P(t)\qquad\text{for a.e. }t\in(0,r(\Omega)).
\end{equation}

The concavity of the distance function, combined with the Brunn–Minkowski inequality for the perimeter (see \cite[Theorem 7.4.5]{schneider}), implies that the map
\[
t\longmapsto P(t)^{\frac{1}{n-1}}
\]
is concave in $[0,r(\Omega)]$. In particular, it is absolutely continuous in $(0,r(\Omega))$ and admits a finite right derivative at $t=0$. Since this function is strictly decreasing, it follows that $P(t)$ is strictly decreasing as well. Moreover, concavity ensures that
\[
\big(P(t)^{\frac{1}{n-1}}\big)''\le 0
\]
in the sense of distributions.

Integrating \eqref{eq:dermu} from $0$ to $t$ and using the fact that $P(s)\le P(\Omega)$ for convex sets, we obtain the estimate
\begin{equation}\label{steiner-measure}
\mu(t)\ge |\Omega|-P(\Omega)\,t \qquad\text{for a.e. }t\in[0,r(\Omega)].
\end{equation}
We recall some bound on the perimeter and the measure of the inner parallel sets proved in \cite{AGS}.

\begin{lemma}
    Let $\Omega$ be a non-empty, bounded, open and convex set of $\R^n$. Then, there exists a positive dimensional constant $c_n$, such that

\begin{equation}
    \label{perimeter_estimate}
    P(t)\leq P(\Omega) - c_n \frac{\abs{\Omega} - \mu(t)}{P^{\frac{1}{n-1}}(\Omega)}.
\end{equation}
\end{lemma}

\begin{lemma}
    Let $\Omega$ be a non-empty, bounded, open and convex set of $\R^n$. Then

\begin{equation}
    \label{measure_estimate1}
    \mu(t)\leq P(\Omega)(r(\Omega)-t)+ \frac{(r(\Omega)-t)^2}{2(n-1)}P'(t).
\end{equation}
\end{lemma}

\subsection{Useful results about the remainder terms}
We define the following two measures of asymmetry
\begin{equation*}
\mathcal{A}(\Omega):=\frac{w_\Omega}{\diam(\Omega)}, \qquad \qquad \text{and} \qquad \qquad  \mathcal{R}(\Omega):=\frac{P(\Omega)r(\Omega)}{\abs{\Omega}}-1.
\end{equation*}
We recall the following bounds, proved in \cite{fenchel_bonnesen} and then generalized in any dimension and for more general operators (see for instance \cite{DBN}. 

\begin{prop}\label{prop:PRM-1}
 Let $\Omega$ be a non-empty, bounded, open and convex set of $\mathbb{R}^n$. Then, 
 \begin{equation}\label{convex_estimates}
1  < \dfrac{P(\Omega) r(\Omega)}{|\Omega|}  \le n. 
 \end{equation}
The lower bound is sharp on a sequence of thinning cylinders, while the upper bound is achieved for instance if $\Omega_{(1-t)R_{\Omega}}= t \Omega $ for $t \in (0,1)$ or on a sequence of thinning pyramids. 
\end{prop}

We also recall a lower bound in \eqref{convex_estimates} in terms of $\mathcal{A}(\Omega)$ proved in \cite{AGS}.

\begin{prop}\label{prop:comparisonasymmetries}
     Let $\Omega$ be a non-empty, bounded, open and convex set of $\R^n$. Then, there exists a positive constant $K=K(n)$ depending only on the dimension of the space, such that
     \begin{equation}\label{eq:PRM-WD}
         \mathcal{R}(\Omega)\ge K(n)\mathcal{A}(\Omega)
     \end{equation}
The exponent of the quantity $\mathcal{A}(\Omega)$ is sharp. A reverse inequality cannot be true, since there are sequences of thinning domains for which the functional $\mathcal{R}(\Omega)$ is not converging to zero (for instance a sequence of thinning triangles in dimension 2).
\end{prop}
Another useful result contained in \cite{AGS} is the following.
\begin{lemma}\label{lemma:M(PM)eP(MP)}
Let $\Omega$ be a non-empty, bounded, open and convex set of $\R^n$. Then,
    \begin{equation}
        \label{eq:lemmaM(mp)}\mu\left(\frac{\abs{\Omega}}{P(\Omega)}\right)\ge q_1(n,\Omega) \abs{\Omega},
    \end{equation}
\begin{equation}\label{eq:LemmaP(MP)}
        P\left(\frac{\abs{\Omega}}{P(\Omega)}\right)\le q_{2}(n,\Omega)P(\Omega),
    \end{equation}
    where
    \begin{equation}
        q_1(n,\Omega)= \frac{\mathcal{R}(\Omega)}{6n}
        \quad \text{ and }\quad q_{2}(n,\Omega)= \left(1+\frac{\mathcal{R}(\Omega)}{n}\right)^{-1}.
    \end{equation}
  
\end{lemma}
\subsection{A one dimensional Laplacian eigenvalue problem}\label{1D_subsec}
In this Section, we study a one-dimensional problem that will be useful to estimate $\lambda_\beta(\Omega)$. We refer to \cite{DP2024, DPP, sperb} for further details. 

We consider the eigenvalue problem in the unknown $X=X(s)$:
\begin{equation}
\label{1dim}
    \begin{cases}
    X''+\nu X=0 \quad\text{in}\ (0,s_0),\\
    X'(0)=0, \\
    X'(s_{0})+\beta X(s_{0})=0, 
    \end{cases}
\end{equation}
where $s_{0}$ is a given positive number.

\begin{teorema}\label{mu_thm}
Let $\beta\ge 0$. Then there exists the smallest  eigenvalue $\nu$ of  \eqref{1dim}, which has the following variational characterization:
\begin{equation*}
\label{def_eigenvalue_mu}
\nu_{1}(\beta,s_0)=\inf_{\substack{v\in W^{1,2}(0,s_{0}) \\ v'(0)=0}} \frac{\int_{0}^{s_{0}}\left|v'(s)\right|^{2}ds+\beta v(s_{0})^{2}}{\int_{0}^{s_{0}}\left|v(s)\right|^{2}ds}.
\end{equation*}
 Moreover, the corresponding eigenfunctions are unique up to a multiplicative constant and have constant sign. The first eigenvalue $\nu_{1}(\beta,s_0)$ is positive. Moreover, the first eigenfunction is
\begin{equation*}
X(s)=\cos\left(\sqrt{\nu_{1}(\beta,s_0)}s \right),\quad s\in (0,s_{0}),
\end{equation*}
and the eigenvalue $\nu_{1}(\beta,s_0)$ is the first positive value that satisfies
\begin{equation}
\label{et}
{\nu}=\frac{\beta^{2}}{\tan^{2}\left(\sqrt{\nu} s_{0} \right)}.
\end{equation}

\end{teorema}

\begin{oss}\label{argument}
It is possible to prove that $X$ is decreasing when $\beta>0$, so $0\leq X\leq 1$. Hence $0\leq \arccos\left(X(s)\right)\leq \frac{\pi}{2}$. Therefore, we have $0\leq \sqrt{\nu_1} s \leq\frac{\pi}{2}$, and then
\[
\nu_{1}(\beta,s_0) \leq \left(\frac{\pi}{2s_{0}}\right)^{2}.
\] 

It is possible to obtain sharper bounds if one consider the following elementary bounds on the $\tan(x)$

$$\frac{2x}{\frac{\pi^2}{4}-x^2}\le \tan(x)\le \frac{\pi^2}{4}\frac{x}{\frac{\pi^2}{4}-x^2}.$$
Indeed, one obtain

\begin{equation}
    \label{bound:mu:s0}
    \frac{\pi^2}{4s_0^2}\frac{1}{1+\frac{\pi^2}{4\beta s_0}}\le \nu_{1}(\beta, s_0)\le \frac{\pi^2}{4s_0^2}\frac{1}{1+ \frac{2}{\beta s_0}}.
\end{equation}

\end{oss}

\section{Upper bound for the Robin Torsion} Concerning the upper bound in Theorem \ref{thm:TlessD}, we prove an estimate of the Robin torsion in terms of the $L^1$ and $L^2$-norms of the distance function from the boundary. 

\begin{proof}[Proof of Theorem \ref{thm:TlessD}]
   We will give the proof in the case of $\Omega$ being a convex polytope, which is defined as the convex hull of finitely many points in $\mathbb R^n$. The proof for a general open, bounded convex set will follow by approximation arguments (see \cite[Theorem $1.8.16$]{schneider}). Let us assume that the polytope $\Omega$ has $N\in\mathbb N$ facets, denoted by $F_i$, $i=1,..., N$, and let us decompose
   \begin{equation*}
      \Omega= \bigcup_{i=1}^N E_i,
   \end{equation*}
   where $E_i$ is a connected component of $\Omega$, where the map $x\to d(x,\partial\Omega)$ is differentiable.\\
   If $v$ is the Robin torsion function in $\Omega$, then denoting by $\nabla_{x'}v= (\partial_1v,...,\partial_{n-1}v)$, where  $x'=(x_1,...,x_{n-1})$ and by $\partial_kv= \frac{\partial v}{\partial x_k}$, $k=1,...,n$, by Cauchy's inequality, we obtain
   \begin{equation*}
   \begin{split}
       T_\beta(\Omega) &=\frac{\left(\displaystyle{\sum_{i=1}^N\int_{E_i} v \, dx }\right)^2}{\displaystyle{\sum_{i=1}^N\bigg[\int_{E_i}(\abs{\nabla_{x'} v}^2+(\partial_nv)^2) \, dx }+\beta \int_{F_i}v^2\,d\mathcal{H}^{n-1}\bigg]}\\
       &\le \sum_{i=1}^N\frac{\displaystyle{\left(\int_{E_i} v \, dx \right)^2}}{\displaystyle{\int_{E_i}(\abs{\nabla_{x'} v}^2+(\partial_nv)^2) \, dx }+\beta \int_{F_i}v^2\,d\mathcal{H}^{n-1}}.  
   \end{split}
   \end{equation*}
   Without loss of generality, we can assume that $F_i$ lies on the $\mathbb R^{n-1}$ hyperplane, so that $E_i$ can be espressed as follows
   \begin{equation*}
       E_i = \{(x',y): x'\in F_i, 0<y<f_i(x')\},\qquad i=1,...,N,
   \end{equation*}
   where $f_i:x'\in F_i\to \mathbb R$ parametrizes $\partial E_i \cap \Omega$. Clearly $F_i = \overline{E_i} \cap \{y=0\}$. In this way
   \begin{equation}\label{eq:TE_i}
   \begin{split}
       T_{\beta, i} :&= \frac{\displaystyle{\left(\int_{E_i} v \, dx \right)^2}}{\displaystyle{\int_{E_i}(\abs{\nabla_{x'} v}^2+(\partial_nv)^2) \, dx }+\beta \int_{F_i}v^2\,dx'}\\
       &\le \frac{\displaystyle{\left(\int_{E_i} v \, dx \right)^2}}{\displaystyle{\int_{E_i}(\partial_nv)^2 \, dx }+\beta \int_{F_i}v^2\,dx'}=\frac{\displaystyle{\left(\int_{F_i}dx'\int_0^{f_i(x')} v \, dx_n \right)^2}}{\displaystyle{\int_{F_i}\bigg[\int_0^{f_i(x')}(\partial_nv)^2 \, dx_n+\beta v^2\bigg]dx' }}.\\
       &\le \int_{F_i}\frac{\displaystyle{\left(\int_0^{f_i(x')} v \, dx_n \right)^2}}{\displaystyle{\int_0^{f_i(x')}(\partial_nv)^2 \, dx_n+\beta v^2 }}\,dx',
       \end{split}
   \end{equation} 
   where, in the last line,  we have used again the Cauchy-Schwarz inequality.
   The term in the last integral is such that
   \begin{equation}
       \frac{\displaystyle{\left(\int_0^{f_i(x')} v \, dx_n \right)^2}}{\displaystyle{\int_0^{f_i(x')}(\partial_nv)^2 \, dx_n+\beta v^2 }}\le T_\beta([0,s]),
   \end{equation}
   with $s=f_i(x')$ and $T_\beta([0,s])$ is the 1D Robin Torsion in the segment $[0,s]$, that is to say
   
   \begin{equation*}
       T_\beta([0,s])=\sup_{g\in H^1([0,s])}\frac{\displaystyle{\bigg(\int_0^sg(t)\,dt\bigg)^2}}{\displaystyle{\int_0^s(g'(t))^2\,dt+\beta g^2(0)}}.
   \end{equation*}
   The Robin Torsion is the $L^1$-norm of the unique solution to the following ODE
   \begin{equation}\label{eq:RobinTorsion1D}
       \begin{cases}
           -\varphi''(t)=1 & \text{in}\; [0,s]\\
           \varphi'(s)=0\\
           \varphi'(0)=\beta \varphi(0).
       \end{cases}
   \end{equation}
   Integrating \eqref{eq:RobinTorsion1D}$_1$, we get
   \begin{equation*}
       \varphi(t)=-\frac{t^2}{2}+c_1t+c_2,
   \end{equation*}
   and imposing the boundary conditions we arrive to
   \begin{equation*}
       \varphi(t)=-\frac{t^2}{2}+st+\frac{s}{\beta}.
   \end{equation*}
   Therefore
   \begin{equation*}
       T_\beta([0,s])= \int_0^s\varphi(t)\,dt = -\frac{s^3}{6}+\frac{s^3}{2}+\frac{s^2}{\beta}=\frac{s^3}{3}+\frac{s^2}{\beta}.
   \end{equation*}
   Combining this with \eqref{eq:TE_i}, we get
   \begin{equation*}
   \begin{split}
        T_{\beta, i} &\le \int_{F_i}\bigg[\frac{f_i(x')^3}{3}+\frac{f_i(x')^2}{\beta}\bigg]\,dx'\\
       &= \int_{F_i}\int_0^{f_i(x')}\bigg(y^2+2\frac{y}{\beta}\bigg)\,dydx'= \int_{E_i}d(x,\partial \Omega)^2\,dx+\frac{2}{\beta}\int_{E_i}d(x,\partial \Omega)\,dx.    
   \end{split} 
   \end{equation*}
   Eventually,
   \begin{equation*}
       T_\beta(\Omega) \le \sum_{i=1}^NT_{\beta,i}\le\int_{\Omega}d(x,\partial \Omega)^2\,dx+\frac{2}{\beta}\int_{\Omega}d(x,\partial \Omega)\,dx.\\
   \end{equation*}
   
\end{proof}

From Theorem \ref{thm:TlessD}, we obtain a generalization of the Makai inequality for the Robin Torsion as follows:

\begin{proof}[Proof of Corollary \ref{thm:geometricmakai}]
    From the proof of Theorem \ref{thm:TlessD} we know that
    \begin{equation*}
        T_{\beta,i}\le \int_{F_i}\bigg[\frac{f_i(x')^3}{3}+\frac{f_i(x')^2}{\beta}\bigg]\,dx'\le \frac{r(\Omega)^2}{3}\int_{F_i}f_i(x')\,dx'+\frac{r(\Omega)}{\beta}\int_{F_i}f_i(x')\,dx'.
    \end{equation*}
    But recalling that
    \begin{equation*}
        \int_{F_i}f_i(x')\,dx'= \int_{F_i}\int_0^{f_i(x')}\,dy\,dx'= \int_{E_i}\,dx = \abs{E_i},
    \end{equation*}
   we have
    \begin{equation*}
        T_{\beta,i}\le \frac{r(\Omega)^2\abs{E_i}}{3}+\frac{r(\Omega)\abs{E_i}}{\beta}.
    \end{equation*}
    Summing up for $i=1,...,n$ and dividing by $r(\Omega)^2\abs{\Omega}$ we get the thesis.
\end{proof}

\subsection{On the optimal sets}\label{subsec:Optimal}
In this subsection we discuss about the optimal sequences of sets that achieves the equality case in inequalities \eqref{lower:torsion:robin}, \eqref{eq:torsiondistancefunction} and \eqref{eq:qualitativemakai}.Let us stress that inequality \eqref{eq:qualitativemakai} is a direct consequence of inequality \eqref{eq:torsiondistancefunction}, meaning that they share the same optimal sequences.\\
Inequalities \eqref{lower:torsion:robin} and \eqref{eq:qualitativemakai} are both optimal on slabs, and one can use either inequalities to prove the sharpness of the other.
 Indeed,  if one wants to prove that \eqref{eq:qualitativemakai} is sharp along a sequence of slab-like domain, we can compose \eqref{eq:qualitativemakai} with \eqref{lower:torsion:robin} 
\begin{equation}\label{eq:optimal1}
\frac{1}{3}+\frac{1}{r(\Omega_\ell)\beta}\ge \frac{T_{\beta}(\Omega_\ell)}{r(\Omega_\ell)^2\abs{\Omega_\ell}}\ge \left(\frac{1}{3}+\frac{1}{r(\Omega_\ell)\beta}\right) \frac{\abs{\Omega_\ell}^2}{P^2(\Omega_\ell)r^2(\Omega_\ell)}.
\end{equation}
and if we send $\ell\to 0$, by Proposition \ref{prop:PRM-1} we have $\frac{\abs{\Omega_\ell}^2}{P^2(\Omega_\ell)r^2(\Omega_\ell)}\to 1$, proving the sharpness.

Analogously, if we want to prove that \eqref{lower:torsion:robin} is sharp on a family of slab-like domains, we can use  \eqref{eq:qualitativemakai}, obtaining
\begin{equation}\label{eq:optimal2}
\frac{1}{3}+\frac{1}{r(\Omega_\ell)\beta}\le \frac{T_{\beta}(\Omega_\ell)P^2(\Omega_\ell)}{\abs{\Omega_\ell}^3}\le \left(\frac{1}{3}+\frac{1}{r(\Omega_\ell)\beta}\right) \frac{P^2(\Omega_\ell)r(\Omega_\ell)^2}{\abs{\Omega_\ell}^2}.    
\end{equation}

Instead, if we consider sequences of thinning cylinders, it is easily seen that
\begin{equation*}
    \lim_{\ell\to +\infty} \frac{T_{\beta}(\Omega_\ell)}{r(\Omega_\ell)^2\abs{\Omega_\ell}}=\lim_{\ell\to +\infty}\frac{T_{\beta}(\Omega_\ell)P^2(\Omega_\ell)}{\abs{\Omega_\ell}^3}= +\infty.
\end{equation*}
While \eqref{eq:optimal1} naturally highlights slab-like domains as optimal candidates, it seems that it does not immediately reveal the role of thinning cylinders as in the Dirichlet case. This is due to the fact that the Robin Torsion does not scale in the standard way, but it becomes visible once we normalize both functionals by $\frac{1}{3}+\frac{1}{r(\Omega_\ell)\beta}$. In fact we have that
\begin{equation*}
1\ge \frac{T_{\beta}(\Omega_\ell)}{r(\Omega_\ell)^2\abs{\Omega_\ell}}\left(\frac{1}{3}+\frac{1}{r(\Omega_\ell)\beta}\right)^{-1}\ge  \frac{\abs{\Omega_\ell}^2}{P^2(\Omega_\ell)r^2(\Omega_\ell)},
\end{equation*}
and
\begin{equation*}
    1\le \frac{T_{\beta}(\Omega_\ell)P^2(\Omega_\ell)}{\abs{\Omega_\ell}^3}\left(\frac{1}{3}+\frac{1}{r(\Omega_\ell)\beta}\right)^{-1}  \le \frac{P^2(\Omega_\ell)r(\Omega_\ell)^2}{\abs{\Omega_\ell}^2}.
\end{equation*}
The right-hand side of both previous inequalies tends to $1$ as $\ell \to \infty$, implying that thinning cylinders are also optimal sequences.

\section{Quantitative estimates for the Robin Torsion}\label{subsec:webtors}

We can now prove the quantitative improvement of the Makai-type inequality for the Robin Torsion.
\begin{proof}[Proof of Theorem \ref{thm:quantitativemakai}]
   Let us recall the upper bound \eqref{eq:torsiondistancefunction},
   $$ T_\beta(\Omega)\le \int_{\Omega}d(x,\partial \Omega)^2\,dx+\frac{2}{\beta}\int_{\Omega}d(x,\partial \Omega)\,dx,$$
and let us start by considering the first term on the right-hand side. We stress that, for this term, the computations can be found in \cite{AGS} and we rewrite them for the sake of completeness. Applying Coarea Formula, integrating by parts and using estimate \eqref{measure_estimate1}, we get
    \begin{equation}\label{eq:makaieq1}
    \begin{aligned}
        \int_{\Omega}d(x,\partial \Omega)^2\,dx&= \int_0^{r(\Omega)}t^2P(t)\,dt= 2\int_0^{r(\Omega)}t\mu(t)\,dt\\&\le 2\int_0^{r(\Omega)}t(r(\Omega)-t)P(t)\,dt+ \frac{1}{n-1}\int_0^{r(\Omega)}t(r(\Omega)-t)^2P'(t)\,dt.
        \end{aligned}
    \end{equation}
    If we integrate by parts the second integral on the right-hand side of \eqref{eq:makaieq1}, we get
    \begin{equation}\label{eq:makaieq2}
    \begin{split}
\int_0^{r(\Omega)}t(r(\Omega)-t)^2P'(t)\,dt&= t(r(\Omega)-t)^2P(t)\bigg|_0^{r(\Omega)}-\int_0^{r(\Omega)}[(r(\Omega)-t)^2-2t(r(\Omega)-t)]P(t)\,dt\\
&=2\int_0^{r(\Omega)}t(r(\Omega)-t)P(t)\,dt-\int_0^{r(\Omega)}(r(\Omega)-t)^2P(t)\,dt.
    \end{split}
    \end{equation}
    We notice that one of the two integrals in \eqref{eq:makaieq2} is equal to the one in \eqref{eq:makaieq1}, therefore
    \begin{equation*}
        \begin{split}\int_0^{r(\Omega)}t^2P(t)\,dt&\le \frac{2n}{n-1}\int_0^{r(\Omega)}t(r(\Omega)-t)P(t)\,dt-\frac{1}{n-1}\int_0^{r(\Omega)}(r(\Omega)-t)^2P(t)\,dt \\
        &=2\frac{n+1}{n-1}r(\Omega)\int_0^{r(\Omega)}tP(t)\,dt-\frac{2n+1}{n-1}\int_0^{r(\Omega)}t^2P(t)\,dt-\frac{r(\Omega)^2\abs{\Omega}}{n-1}.
        \end{split}
    \end{equation*}
    Summing up the same terms, we have
    \begin{equation}\label{eq:makaieq3}
        \int_0^{r(\Omega)}t^2P(t)\,dt\le \frac{2(n+1)}{3n}r(\Omega) \int_{0}^{r(\Omega)}tP(t)\,dt-\frac{r(\Omega)^2\abs{\Omega}}{3n}.
    \end{equation}
    We now estimate the integral on the right-hand side of \eqref{eq:makaieq3}. Integrating by parts and using again \eqref{eq:dermu} and \eqref{measure_estimate1}
    \begin{equation*}
    \begin{split}
\int_{0}^{r(\Omega)}tP(t)\,dt= \int_{0}^{r(\Omega)}\mu(t)\,dt&\le \int_{0}^{r(\Omega)}(r(\Omega)-t)P(t)\,dt+ \frac{1}{2(n-1)}\int_{0}^{r(\Omega)}(r(\Omega)-t)^2P'(t)\,dt\\
&=\frac{n}{n-1}\int_{0}^{r(\Omega)}(r(\Omega)-t)P(t)\,dt-\frac{r(\Omega)^2P(\Omega)}{2(n-1)}\\
&=\frac{n}{n-1}r(\Omega)\abs{\Omega}-\frac{n}{n-1}\int_{0}^{r(\Omega)}tP(t)\,dt-\frac{r(\Omega)^2P(\Omega)}{2(n-1)}.
\end{split}
    \end{equation*}
    Therefore
    \begin{equation}\label{eq:makaieq4}
        \int_{0}^{r(\Omega)}tP(t)\,dt\le \frac{n}{2n-1}r(\Omega)\abs{\Omega}-\frac{r(\Omega)^2P(\Omega)}{2(2n-1)}
    \end{equation}
    Combining  \eqref{eq:makaieq3} and \eqref{eq:makaieq4}, we get
    \begin{equation*}
       \begin{split}
       \int_0^{r(\Omega)}t^2P(t)\,dt&\le \frac{2(n+1)}{3(2n-1)}r(\Omega)^2\abs{\Omega}-\frac{n+1}{3n(2n-1)} r(\Omega)^3P(\Omega)-\frac{r(\Omega)^2\abs{\Omega}}{3n}\\
       &=\frac{r(\Omega)^2\abs{\Omega}}{3}+\frac{n+1}{3n(2n-1)}\bigg[r(\Omega)^2\abs{\Omega}-r(\Omega)^3P(\Omega)\bigg]\\
       &=\frac{r(\Omega)^2\abs{\Omega}}{3}-\frac{n+1}{3n(2n-1)}r(\Omega)^2\abs{\Omega}\mathcal{R}(\Omega).
\end{split}
\end{equation*}
With regards to the second integral, applying Coarea formula and integrating by parts, we get
\begin{equation*}
    \int_{\Omega}d(x,\partial \Omega)\,dx = \int_0^{r(\Omega)}tP(t)\,dt = \int_0^{r(\Omega)}\mu(t)\,dt.
\end{equation*}
Recalling that $\mu(t)\le P(t)(r(\Omega)-t)$, we have that
\begin{equation*}
     \int_0^{r(\Omega)}tP(t)\,dt=\int_0^{r(\Omega)}\mu(t)\,dt \le r(\Omega)\abs{\Omega}-\int_{0}^{r(\Omega)}tP(t)\,dt,
\end{equation*}
so that
\begin{equation*}
    \int_0^{r(\Omega)}tP(t)\,dt \le\frac{r(\Omega)\abs{\Omega}}{2}.
\end{equation*}
Therefore
\begin{equation*}
    \frac{2}{\beta}\int_{\Omega}d(x,\partial \Omega)\,dx\le\frac{r(\Omega)\abs{\Omega}}{\beta},
    \end{equation*}
and
    \begin{equation*}
        \left(\frac{1}{3}+\frac{1}{r(\Omega)\beta}\right)-\frac{T_\beta(\Omega)}{r(\Omega)^2\abs{\Omega}}\ge \frac{n+1}{3n(2n-1)}\mathcal{R}(\Omega).
    \end{equation*}

    \vspace{2mm}
    Let us now prove the upper bound. If we multiply and divide the functional by $P^2(\Omega)/\abs{\Omega}^2$ and use the lower bound for the P\'olya functional \eqref{lower:torsion:robin}, we get
\begin{equation*}
\begin{split}
\frac{1}{3}+\frac{1}{r(\Omega)\beta}-\frac{T_\beta(\Omega)}{r(\Omega)^2 \abs{\Omega}}&= \frac{1}{3}+\frac{1}{r(\Omega)\beta}-\frac{T_\beta(\Omega)P^2(\Omega)}{ \abs{\Omega}^3}\frac{\abs{\Omega}^2}{r(\Omega)^2 P^2(\Omega)} \\&\le\left(\frac{1}{3}+\frac{1}{r(\Omega)\beta}\right)\bigg(1-\frac{\abs{\Omega}^2}{r(\Omega)^2 P^2(\Omega)}\bigg)\\
    &= \left(\frac{1}{3}+\frac{1}{r(\Omega)\beta}\right)\bigg(1+\frac{\abs{\Omega}}{r(\Omega) P(\Omega)}\bigg)\bigg(1-\frac{\abs{\Omega}}{r(\Omega) P(\Omega)}\bigg)\\&\le 2\left(\frac{1}{3}+\frac{1}{r(\Omega)\beta}\right)\bigg(\frac{P(\Omega)r(\Omega)}{\abs{\Omega}}-1\bigg).
        \end{split}
\end{equation*}
    \end{proof}
We proceed to the proof of Theorem \ref{thm:quantitativepolyatorsion}, starting with the lower bound.

\begin{proof}[Proof of Theorem \ref{thm:quantitativepolyatorsion} ]

We start by proving the lower bound in \eqref{polya:lower:R}.
Let us consider as a test function in problem \eqref{torsion:robin} $f(x)= g(d(x,\partial\Omega))$, where $g$ is suitably chosen. At this point, coarea formula and an integration by parts allow us to write 
$$
\int_\Omega f(x)\, dx = \int_0^{r(\Omega)} g(t) P(t) \, dt=\int_0^{r(\Omega)} g'(t) \mu(t) \, dt+\abs{\Omega}g(0),$$
and
\begin{equation*}
 \beta \int_{\partial\Omega} f^2\, d\mathcal{H}^{n-1}=\beta g^2(0)P({\Omega}), \quad \quad   \int_\Omega \abs{ \nabla f}^2\,dx = \int_0^{r(\Omega)} g'^2(t) P(t) \, dt.
\end{equation*} 
In this way, we get

\begin{equation*}
    \label{torsionpolya}
    T_\beta(\Omega) \geq \frac{\displaystyle{\left(\int_\Omega f(x)\, dx\right)^2}}{\displaystyle{\int_\Omega \abs{ \nabla f}^2\, dx+\beta \int_{\partial\Omega} f^2\, d\mathcal{H}^{n-1}}} =\frac{\displaystyle{\left(\int_0^{r(\Omega)} g'(t) \mu(t) \, dt+\abs{\Omega}g(0)\right)^2}}{\displaystyle{\int_0^{r(\Omega)} g'^2(t) P(t) \, dt}+\beta g^2(0)P(\Omega)},
\end{equation*}
and choosing $g(t) = \displaystyle{\int_0^t\frac{\mu(s)}{P(s)}\, ds+ \frac{\abs{\Omega}}{\beta P(\Omega)}}$, we finally have \begin{equation}\label{eq:lowboundpolya}
T_\beta(\Omega) \geq  \int_0^{r(\Omega)}\frac{\mu^2(t)}{P(t)}\,dt+\frac{\abs{\Omega}^2}{\beta P(\Omega)}.
\end{equation}
To obtain our quantitative result, we must split the integral on the right hand-side of \eqref{eq:lowboundpolya} into the two intervals $[0,\Bar{t}]$ and $[\Bar{t},r(\Omega)]$, where $\Bar{t}=\abs{\Omega}/P(\Omega)<r(\Omega)$. By \eqref{steiner-measure} we get
  \begin{align*}
     & \int_{0}^{r(\Omega)}\dfrac{\mu^2(t)}{P(t)}dt\\
      &\geq \dfrac{1}{P^2(\Omega)}\int_0^{\frac{|\Omega|}{P(\Omega)}}\left(|\Omega|-P(\Omega)t\right)^2 P(\Omega) dt
      +\dfrac{1}{P^2(\Omega)}\int_{\frac{|\Omega|}{P(\Omega)}}^{r(\Omega)} \mu^2(t)(-\mu'(t))dt \\
      & =\dfrac{|\Omega|^3}{3 P^2(\Omega)}+\dfrac{1}{P^2(\Omega)}\dfrac{\mu^3\left(\dfrac{|\Omega|}{P(\Omega)}\right)}{3},
  \end{align*}
  
  Therefore, multiplying by $P^2(\Omega)/\abs{\Omega}^3$, and using the fact that $P(\Omega)/\abs{\Omega}> r(\Omega)^{-1}$, we get
  \begin{equation*}
    \dfrac{T_\beta(\Omega)P^2(\Omega)}{|\Omega|^3}\geq \dfrac{1}{3}+\frac{P(\Omega)}{\beta\abs{\Omega}}+
      \dfrac{1}{3}\dfrac{\mu^3\left(\dfrac{|\Omega|}{P(\Omega)}\right)}{|\Omega|^3}> \dfrac{1}{3}+\frac{1}{r(\Omega)\beta}+
      \dfrac{1}{3}\dfrac{\mu^3\left(\dfrac{|\Omega|}{P(\Omega)}\right)}{|\Omega|^3}.  
  \end{equation*}
  Rearranging the terms and applying Lemma \ref{lemma:M(PM)eP(MP)}, we obtain 
\begin{align*}
      \dfrac{T_\beta(\Omega)P^2(\Omega)}{|\Omega|^3}-\bigg(\dfrac{1}{3}+\frac{1}{r(\Omega)\beta}\bigg)\geq
      \dfrac{1}{3}\dfrac{\mu^3\left(\dfrac{|\Omega|}{P(\Omega)}\right)}{|\Omega|^3}\geq \dfrac{q_1(n,\Omega)^3}{3}=\dfrac{1}{2^3\cdot3^4  n^3}\left(\frac{P(\Omega) r(\Omega)}{\abs{\Omega}}-1\right)^3,
  \end{align*}
   and this proves the lower bound in \eqref{polya:lower:R}.
   Now, regarding the upper bound, we just multiply and divide by $r(\Omega)^2$ and use the Robin Makai inequality \eqref{eq:qualitativemakai}, having
   \begin{equation*}
       \begin{split}
       \dfrac{T_\beta(\Omega)P^2(\Omega)}{|\Omega|^3}&-\bigg(\dfrac{1}{3}+\frac{1}{r(\Omega)\beta}\bigg)=\dfrac{T_\beta(\Omega)}{r(\Omega)^2\abs{\Omega}}\dfrac{P^2(\Omega)r(\Omega)^2}{|\Omega|^2}-\bigg(\dfrac{1}{3}+\frac{1}{r(\Omega)\beta}\bigg)\\
       &\le 
       \bigg(\dfrac{1}{3}+\frac{1}{r(\Omega)\beta}\bigg)\dfrac{P^2(\Omega)r(\Omega)^2}{|\Omega|^2}-\bigg(\dfrac{1}{3}+\frac{1}{r(\Omega)\beta}\bigg)\\
       &=\bigg(\dfrac{1}{3}+\frac{1}{r(\Omega)\beta}\bigg)\bigg[\dfrac{P^2(\Omega)r(\Omega)^2}{|\Omega|^2}-1\bigg]\\
       &=\bigg(\dfrac{1}{3}+\frac{1}{r(\Omega)\beta}\bigg)\bigg(\dfrac{P(\Omega)r(\Omega)}{|\Omega|}+1\bigg)\mathcal{R}(\Omega)\le (n+1)\bigg(\dfrac{1}{3}+\frac{1}{r(\Omega)\beta}\bigg)\mathcal{R}(\Omega).
       \end{split}
   \end{equation*}
   \end{proof}

\section{Quantitative estimate for the  Robin eigenvalue}
In this section, we prove Theorem \ref{thm:eigen_R}, which ensures that the optima in inequality
$$\nu_1(\beta, s_0)\ge \lambda_\beta(\Omega)$$
are the same as in the inequality $\mathcal{R}(\Omega)\ge 0$.
\begin{proof}[Proof of Theorem \ref{thm:eigen_R}]
 The first lines of the proof follow the same argument proposed in \cite{DP2024}, whose computations are analogous to the one shown in Subsection \ref{subsec:webtors}. Let us use as a test function in the variational characterization of $\lambda_\beta(\Omega)$ the function $f(x)=g(t)$, where $g$ depends only on the distance function from the boundary of $\Omega$. Then by coarea formula we get

\begin{equation}
\label{polyalambda}
    \lambda_\beta(\Omega) \leq \dfrac{\displaystyle{\int_0^{r(\Omega)} (g'(t))^2 P(t) \, dt}+\beta g(0)^2P(\Omega)}{\displaystyle{\int_0^{r(\Omega)} g^2(t) P(t) \, dt}}.
\end{equation}
The latest, with the change of  variables $s= \frac{\mu(t)}{P(\Omega)}$, leads to 

\begin{equation*}
    \lambda_\beta(\Omega)\le\frac{\displaystyle{\int_{0}^{\frac{\abs{\Omega}}{P(\Omega)}}h'(s)^2 P(t(s))^2\,ds+\beta P^2(\Omega)}h^2\left(\frac{\abs{\Omega}}{P(\Omega)}\right)}{P^2(\Omega)\displaystyle{\int_0^{\frac{\abs{\Omega}}{P(\Omega)}}h(s)^2\,ds}}
\end{equation*}
where $h(s)=g(t)$. Now, we choose $\Bar{t}= \frac{\abs{\Omega}}{P(\Omega)}$ and we  denote by 
\begin{equation}\label{eq:bars}
    \Bar{s}=\frac{\mu(\Bar{t})}{P(\Omega)}.
    \end{equation}
Hence, we divide the integral at the numerator in \eqref{polyalambda} at $\Bar{s}$, obtaining

\begin{equation*}
\begin{split}
\lambda_{\beta}(\Omega) &\le \frac{\displaystyle \int_{0}^{\Bar{s}}h'(s)^2P(t(s))^2\,ds +\int_{\Bar{s}}^{\frac{\abs{\Omega}}{P(\Omega)}}h'(s)^2P(t(s))^2\,ds+\beta P^2(\Omega)h^2\left(\frac{\abs{\Omega}}{P(\Omega)}\right)}{P^2(\Omega)\displaystyle\int_0^{\frac{\abs{\Omega}}{P(\Omega)}}h(s)^2\,ds}\\
    &\le\frac{\displaystyle P(\Omega)\int_{0}^{\Bar{s}}h'(s)^2P(t(s))\,ds +P(\Omega)^2\int_{\Bar{s}}^{\frac{\abs{\Omega}}{P(\Omega)}}h'(s)^2\,ds+P^2(\Omega)\beta h\left(\frac{\abs{\Omega}}{P(\Omega)}\right)}{P^2(\Omega)\displaystyle\int_0^{\frac{\abs{\Omega}}{P(\Omega)}}h(s)^2\,ds}.
\end{split}
\end{equation*}
Let us observe that if $s\in[0,\overline{s}]$, $P(t(s))\le P(\overline{t})$ thanks to the monotonicity of the perimeter, we can apply Lemma \ref{lemma:M(PM)eP(MP)}, obtaining
\begin{equation}\label{eq:estimateintermediate}
\begin{split}
    \lambda_{\beta}(\Omega) &\le\frac{\displaystyle q_2(n,\Omega)P^2(\Omega)\int_{0}^{\Bar{s}}h'(s)^2\,ds +P^2(\Omega)\int_{\Bar{s}}^{\frac{\abs{\Omega}}{P(\Omega)}}h'(s)^2\,ds+P^2(\Omega)\beta h\left(\frac{\abs{\Omega}}{P(\Omega)}\right)}{P^2(\Omega)\displaystyle\int_0^{\frac{\abs{\Omega}}{P(\Omega)}}h(s)^2\,ds}\\
    &=\frac{\displaystyle (q_2(n,\Omega)-1)\int_{0}^{\Bar{s}}h'(s)^2\,ds +\int_{0}^{\frac{\abs{\Omega}}{P(\Omega)}}h'(s)^2\,ds+\beta h\left(\frac{\abs{\Omega}}{P(\Omega)}\right)}{\displaystyle\int_0^{\frac{\abs{\Omega}}{P(\Omega)}}h(s)^2\,ds}.
\end{split}
\end{equation}
Now, let us denote by
$$s_0=\frac{\abs{\Omega}}{P(\Omega)},$$ and let us choose $h(s)=\cos(\sqrt{\nu_{1}}s)$, for $0\le s\le s_0$, we have
\begin{equation*}
    \int_0^{s_0}h(s)^2\,ds=\frac{\abs{\Omega}}{2P(\Omega)}+\frac{\sin\left(2\sqrt{\nu_{1}}\frac{\abs{\Omega}}{P(\Omega)}\right)}{4\sqrt{\nu_{1}}}\leq \dfrac{\abs{\Omega}}{P(\Omega)}.
\end{equation*}
In this way, we get
\begin{equation*}
   \nu_{1}\left(\beta,s_0\right) -\lambda_\beta(\Omega)\ge(1-q_2(n,\Omega))\frac{P(\Omega)}{\abs{\Omega}}\nu_{1}\left(\beta,s_0 \right)\int_0^{\Bar{s}}\sin^2(\sqrt{\nu_{1}}s)\,ds. 
\end{equation*}
Using the inequality $\sin(r)\ge \frac{2}{\pi}r$, which is valid for every $r\in [0,\pi/2]$, we have 
\begin{equation}\label{eq:estimatebars3}
     \nu_{1}\left(\beta,s_0\right) -\lambda_\beta(\Omega)\ge\frac{4}{3\pi^2}(1-q_2(n,\Omega))\frac{P(\Omega)}{\abs{\Omega}}\nu_{1}^2\Bar{s}^3.
\end{equation}
Recalling \eqref{eq:bars} and Lemma \ref{lemma:M(PM)eP(MP)}, we have that
\begin{equation}\label{eq:estimatebars}
    \Bar{s}^3=\frac{\mu\left(\frac{\abs{\Omega}}{P(\Omega)}\right)^3}{P^3(\Omega)}\ge\frac{\abs{\Omega}^3}{ 6^3 n^3P(\Omega)^3}\left(\frac{P(\Omega)r(\Omega)}{|\Omega|}-1\right)^3.
\end{equation}
Moreover, by the definition of $q_2(n,\Omega)$ and Proposition \eqref{convex_estimates}, we get
\begin{equation}\label{eq:estimateintermediate2}
    1-q_2(n,\Omega)= 1-\frac{1}{1+\frac{1}{n}\left(\frac{P(\Omega)r(\Omega)}{\abs{\Omega}}-1\right)}= \frac{\frac{1}{n}\left(\frac{P(\Omega)r(\Omega)}{\abs{\Omega}}-1\right)}{1+\frac{1}{n}\left(\frac{P(\Omega)r(\Omega)}{\abs{\Omega}}-1\right)}\ge \frac{1}{2n-1}\left(\frac{P(\Omega)r(\Omega)}{\abs{\Omega}}-1\right).
\end{equation}
Combining \eqref{eq:estimatebars3}, \eqref{eq:estimatebars}, and \eqref{eq:estimateintermediate2}, we have
\begin{equation*}
     \nu_{1}\left(\beta,s_0\right) -\lambda_\beta(\Omega)\ge\frac{1}{2\cdot 3^4 \pi (2n-1)n^3}\cdot\frac{\abs{\Omega}^2}{P(\Omega)^2}\nu_{1}^2\left(\frac{P(\Omega)r(\Omega)}{\abs{\Omega}}-1\right)^4.
\end{equation*}

Now let us observe that $\mu$ is the first positive root of
$$\sqrt{\nu_{1}}\tan(\sqrt{\nu_{1}}s_0)=\beta,$$ that combined with the inequality \eqref{bound:mu:s0}
gives
$$\nu_{1}s_0^2\ge\left(\frac{\frac{\pi^2}{4}}{1+\frac{\pi^2P(\Omega)}{4   \abs{\Omega}\beta}}\right)\ge \left(\frac{\frac{\pi^2}{4}}{1+\frac{\pi^2n}{4   \beta r(\Omega)}}\right)  $$
\begin{equation*}
     \frac{ \nu_{1}\left(\beta,s_0\right) \abs{\Omega}^2 }{P^2(\Omega)} -\frac{\lambda_\beta(\Omega)\abs{\Omega}^2}{P^2(\Omega)}\ge\frac{1}{2\cdot 3^4 \pi (2n-1)n^3}\cdot\left(\frac{\frac{\pi^2}{4}}{1+\frac{\pi^2n}{4   \beta r(\Omega)}}\right)  ^2\left(\frac{P(\Omega)r(\Omega)}{\abs{\Omega}}-1\right)^4,
\end{equation*}
which concludes the proof of the lower bound.

\vspace{2mm}
We now  prove the upper bound in \eqref{stima:autoval:R}.
  Let $u_{\Omega}$ be the solution to 
  \begin{equation}
      \begin{cases}
          -\Delta u_{\Omega}=1 & \textrm{in}\;\Omega\\
          u_{\Omega}=0 &\textrm{on}\;\partial \Omega
      \end{cases}
  \end{equation}
  and let $M(\Omega)=\norma{u_{\Omega}}_{\infty}$. In \cite{ilaria}, it was proved that 
  $$2M(\Omega)\geq 3\dfrac{T(\Omega)}{\abs\Omega}\geq \dfrac{\abs{\Omega}^2}{P^2(\Omega)},$$
  hence, if we define $t_0=\sqrt{2M(\Omega)}$, we have $t_0>s_0$. In \cite{DPP}, it was proved that $\lambda_{\beta}(\Omega)\geq \nu_{1}(\beta, t_0)$, hence
  $$0\leq \nu_{1}(\beta,s_0)-\lambda_{\beta}(\Omega)\leq \nu_{1}(\beta,s_0)-\nu_{1}(\beta, t_0),$$ so we aim to prove that the right-hand side is bounded from abvove by a constant times $\mathcal{R}(\Omega)$.\\ 
  We denote $\nu_1(s_0)=\nu_{1}(\beta,s_0)$, $\nu_1(t_0)=\nu_{1}(\beta,t_0)$.
  In the variational definition of $\nu_{1}(\beta,s_0)$, we use as a test function $X=\cos(\sqrt{\nu_{1}(t_0)}s)$, so we obtain
  \[
\nu_{1}(s_0)-\nu_{1}(t_0)\leq \nu_{1}(t_0)\left[ \dfrac{\displaystyle{\int_0^{s_0}\sin^2(\sqrt{\nu_{1}(t_0)}s)ds}}{\displaystyle{\int_0^{s_0}\cos^2(\sqrt{\nu_{1}(t_0)}s)ds}}-1\right]+\dfrac{\displaystyle{\beta\cos(\sqrt{\nu_{1}(t_0)}s_0)}}{\int_0^{s_0}\cos^2(\sqrt{\nu_{1}(t_0)}s)}.
  \]

  We can explicitly compute

  $$\int_0^{s_0}\sin^2(\sqrt{\nu_{1}(t_0)}s)=\dfrac{s_0}{2}-\dfrac{\sin(\sqrt{\nu_{1}(t_0)}s_0)\cos(\sqrt{\nu_{1}(t_0)}s_0)}{2\sqrt{\nu_{1}(t_0)}}$$
  $$\int_0^{s_0}\cos^2(\sqrt{\nu_{1}(t_0)}s)=\dfrac{s_0}{2}+\dfrac{\sin(\sqrt{\nu_{1}(t_0)}s_0)\cos(\sqrt{\nu_{1}(t_0)}s_0)}{2\sqrt{\nu_{1}(t_0)}}$$

  obtaining
 
 \begin{equation}
 \label{stima_mu_t_zero}
  \nu_{1}(s_0)-\nu_{1}(t_0)\leq \dfrac{\cos(\sqrt{\nu_{1}(t_0)}s_0)\left[ \beta\cos(\sqrt{\nu_{1}(t_0)}s_0)-\sqrt{\nu_{1}(t_0)}\sin(\sqrt{\nu_{1}(t_0)}s_0)\right]}{\dfrac{s_0}{2}+\dfrac{\sin(\sqrt{\nu_{1}(t_0)}s_0)\cos(\sqrt{\nu_{1}(t_0)}s_0)}{2\sqrt{\nu_{1}(t_0)}}}.
 \end{equation}
 Let us multiply and divide the numerator in \eqref{stima_mu_t_zero} by $\sqrt{\beta^2+\nu_1(t_0)}$ and let be $\gamma \in \mathbb  R$ such that
$$\sin\gamma=\dfrac{\beta}{\sqrt{\beta^2+\nu_{1}(t_0)}}, \qquad \cos\gamma=\dfrac{\sqrt{\nu_{1}(t_0)}}{\sqrt{\beta^2+\nu_{1}(t_0)}}.$$
In this way, the term in the squared brackets becomes
\begin{equation*}
\begin{split}
    \frac{\beta}{\sqrt{\beta^2+\nu_{1}(t_0)}}&\cos(\sqrt{\nu_{1}(t_0)}s_0)-\sqrt{\dfrac{\nu_{1}(t_0)}{\beta^2+\nu_{1}(t_0)}}\sin(\sqrt{\nu_{1}(t_0)}s_0)\\
    &= \sin\gamma\cos(\sqrt{\nu_{1}(t_0)}s_0)-\cos\gamma \sin(\sqrt{\nu_{1}(t_0)}s_0)= \sin(\gamma -\sqrt{\nu_1(t_0)}s_0).
    \end{split}
\end{equation*}
  In particular, from \eqref{et}, we can choose $\gamma \in [0,\pi/2]$ so that $\gamma=\sqrt{\nu_{1}(t_0)}t_0$. Hence, we can rewrite \eqref{stima_mu_t_zero} as follows
  $$  \nu_{1}(s_0)-\nu_{1}(t_0)\leq \dfrac{\sqrt{\beta^2+\nu_{1}(t_0)}\sin(\sqrt{\nu_{1}(t_0)}(t_0-s_0))}{\dfrac{s_0}{2}+\dfrac{\sin(\sqrt{\nu_{1}(t_0)}s_0)\cos(\sqrt{\nu_{1}(t_0)}s_0)}{2\sqrt{\nu_{1}(t_0)}}}\leq 2\sqrt{\beta^2+\nu_{1}(t_0)}\sqrt{\nu_{1}(t_0)}\dfrac{(t_0-s_0)}{s_0}.$$
  From \cite{DPP}, we know that $t_0\leq r(\Omega)$, hence 
  $$\nu_{1}(s_0)-\nu_{1}(t_0)\leq 2\sqrt{\beta^2+\nu_{1}(t_0)}\sqrt{\nu_{1}(t_0)} \mathcal{R}(\Omega).$$

Let us now multiply by $s_0^2$. Since $t_0>s_0$, then $\nu_{1}(t_0)\le \nu_{1}(s_0)$. Moreover we know that $\sqrt{\nu_{1}(s_0)}s_0\le \pi/2$, therefore
\begin{equation*}
2\sqrt{\beta^2+\nu_{1}(t_0)}\sqrt{\nu_{1}(t_0)}s_0^2\le 2\nu_{1}(s_0)\sqrt{1+\frac{\beta^2}{\nu_{1}(s_0)}}s_0^2\le \frac{\pi^2}{2}\sqrt{1+\frac{\beta^2}{\nu_{1}(s_0)}}.
\end{equation*}
Now, using \eqref{bound:mu:s0} and the fact that $s_0< r(\Omega)$, we have that
\begin{equation*}
    \nu_{1}(s_0)\ge \frac{\pi^2}{4}\frac{1}{r(\Omega)^2+\frac{\pi^2}{4\beta}r(\Omega)},
\end{equation*}
so that
\begin{equation*}
    \nu_{1}(s_0)s_0^2-\nu_{1}(t_0)s_0^2\le \frac{\pi^2}{2}\sqrt{1+\frac{4\beta^2}{\pi^2}(r(\Omega)^2+\frac{\pi^2}{4\beta}r(\Omega))}\mathcal{R}(\Omega),
\end{equation*}
completing the proof.
\end{proof}

\section{Quantitative inequalities involving $\mathcal{A}(\Omega)$}
\label{section_alpha}
We conclude with the two theorem in which we show that is possible to improve inequalities \eqref{polya:lower:R} and \eqref{eigen:polya:thm} by adding the sharp power of the asymmetry $\mathcal{A}(\Omega)$. 
 \begin{prop} \label{prop:quantitative_wd}
Let $\Omega$ be a non-empty, bounded, open and convex set of $\R^n$. Then, 
\begin{equation}\label{quantitative_width}
 \frac{T_\beta(\Omega) P^2(\Omega)}{\abs{\Omega}^3}-\left(\frac{1}{3}+\frac{1}{\beta r(\Omega)}\right)  \ge C_4(n) \mathcal{A}(\Omega),
\end{equation}
where $C_4(n)$ is a positive constant depending only on the dimension of the space $n$.
\end{prop} 

 \begin{proof}[Proof of Proposition \ref{prop:quantitative_wd}]
 Let us recall the  lower bound in \eqref{eq:lowboundpolya}
 
 \begin{equation*}
   T_\beta(\Omega) \geq  \int_0^{r(\Omega)}\frac{\mu^2(t)}{P(t)}\,dt+\frac{\abs{\Omega}^2}{\beta P(\Omega)},
\end{equation*}
and let us divide the integral above at the value $\overline{t}$ defined for some $\Tilde{c}\in (0,1)$ as
\begin{equation}
    \label{labelcomega}
    \mu(\overline{t})= \Tilde{c}\abs{\Omega}.
\end{equation}

Moreover, \eqref{perimeter_estimate} and the monotonicity of $\mu$ gives, for all $t\in [\overline{t}, r(\Omega)]$
$$P(t)\le P(\Omega)- c_n \frac{\abs{\Omega}-\mu(\overline{t})}{P^\frac{1}{n-1}(\Omega)}$$
and together with \eqref{labelcomega}, we can write
\begin{align*}
    \int_0^{r(\Omega)}\frac{\mu^2(t)}{P(t)} \, dt&\geq \int_0^{\overline{t}}\frac{\mu^2(t)}{P(t)} \, dt+\int_{\overline{t}}^{r(\Omega)}\frac{\mu^2(t)}{P(t)} \, dt \\
    &\geq  \frac{1}{P^2(\Omega)} \int_0^{\overline{t}}\mu^2(t)(-\mu'(t))\, dt+ \frac{1}{\displaystyle{P(\Omega)\left(P(\Omega) - c_n \frac{\abs{\Omega} - \mu(\overline{t})}{P^{\frac{1}{n-1}}(\Omega)}\right)}}\int_{\overline{t}}^{r(\Omega)}\mu^2(t) (-\mu'(t))\, dt\\
    &= \frac{1}{P^2(\Omega)} \frac{\abs{\Omega}^3-\mu^3(\overline{t}) }{3}+ \frac{1}{\displaystyle{P(\Omega)\left(P(\Omega) - c_n \frac{\abs{\Omega} - \mu(\overline{t})}{P^{\frac{1}{n-1}}(\Omega)}\right)}} \frac{\mu^3(\overline{t}) }{3}\\
    & \geq \frac{1}{P^2(\Omega)} \frac{\abs{\Omega}^3-\mu^3(\overline{t}) }{3}+ \frac{1}{P^2(\Omega)}\left(1 +c_n\frac{\abs{\Omega} - \mu(\overline{t})}{P^{\frac{n}{n-1}}(\Omega)}\right) \frac{\mu^3(\overline{t}) }{3}\\
    &= \frac{\abs{\Omega}^3}{3P^2(\Omega)}+  c_n\frac{(1-\Tilde{c})\Tilde{c}^3}{P^{2+\frac{n}{n-1}}(\Omega)}\frac{\abs{\Omega}^4 }{3}.
\end{align*}
Now we choose $\Tilde{c}$ in order to maximize $(1-\Tilde{c})\Tilde{c}^3$. So we find the maximum in $(0,1)$ of the function $f(x)=(1-x)x^3$, which gives
\begin{equation*}
    \Tilde{c}= \frac{3}{4}.
\end{equation*}
Hence, we have
\begin{align}\label{eq:estimate1}
   \frac{ T_\beta(\Omega)P^2(\Omega)}{\abs{\Omega}^3}&\geq \left(\frac{1}{3}+\frac{1}{\beta r(\Omega)}\right)+ \frac{27 c_n}{256} \frac{\abs{\Omega}}{P(\Omega)}\frac{1 }{P^{\frac{1}{n-1}}(\Omega)} \geq \left(\frac{1}{3}+\frac{1}{\beta r(\Omega)}\right)+ \frac{27 c_n}{256n}\frac{r(\Omega)}{P^{\frac{1}{n-1}}(\Omega)}
\end{align}

Combining \eqref{eq:estimate1} with \eqref{eq:lowboundinradius} and \eqref{eq:perdiam}, we get the thesis.

\end{proof}

Moreover, we managed to prove a sharp result even for the P\'olya eigenvalue functional.
\begin{prop} 
\label{prop:quantitative_wd2}
Let $\Omega$ be a non-empty, bounded, open and convex set of $\R^n$. Then, 
\begin{equation}\label{quantitative_width2}
 \frac{\pi^2}{4}\frac{1}{1+\frac{2}{\beta r(\Omega)}} - \frac{\lambda( \Omega) \abs{\Omega}^2}{P^2(\Omega)}
 \ge C_5(n) \mathcal{A}(\Omega),
\end{equation}
where $C_5(n)$ is a positive constant depending only on the dimension of the space $n$.
\end{prop}

\begin{proof}[Proof of Proposition \ref{prop:quantitative_wd2}]
We start from \eqref{polyalambda} and we perform the change of variables   $s= \frac{\mu(t)}{P(\Omega)}$ as in  Theorem \ref{thm:quantitativepolyaeigenvalue}

\begin{equation}
    \label{lambda:change}\lambda_\beta(\Omega)\le\frac{\displaystyle{\int_{0}^{\frac{\abs{\Omega}}{P(\Omega)}}h'(s)^2 P(t(s))^2\,ds+\beta P^2(\Omega)}h^2\left(\frac{\abs{\Omega}}{P(\Omega)}\right)}{P^2(\Omega)\displaystyle{\int_0^{\frac{\abs{\Omega}}{P(\Omega)}}h(s)^2\,ds}}
\end{equation}

Let us start by observing that if $s\in \left[0,\frac{\abs{\Omega}}{2P(\Omega)}\right]$, then $t(s)\in [\overline{t}, r(\Omega)]$, where $\mu(\overline{t})=\abs{\Omega}/2$, and so by \eqref{perimeter_estimate} we get

\begin{equation*}
    P(t(s))\le P(\Omega)-c_n\frac{\abs{\Omega}}{2P^\frac{1}{n-1}(\Omega)}, \quad \quad \forall s \in \left[0,\frac{\abs{\Omega}}{2P(\Omega)} \right].
\end{equation*}

We can now split the integral at the numerator in \eqref{lambda:change} at the value $\frac{s_0}{2}=\frac{\abs{\Omega}}{2P(\Omega)}$, obtaining 
\begin{equation}
    \label{gradient:term}
    \begin{aligned}
        \int_{0}^{s_0}h'(s)^2 P(t(s))^2\,ds= \int_{0}^{\frac{s_0}{2}}h'(s)^2 P(t(s))^2\,ds+\int_{\frac{s_0}{2}}^{s_0}h'(s)^2 P(t(s))^2\,ds \le \\
        P(\Omega)\left(P(\Omega)-c_n\frac{\abs{\Omega}}{2P^\frac{1}{n-1}(\Omega)}\right)\int_{0}^{\frac{s_0}{2}}h'(s)^2 \,ds+ P^2(\Omega)\int_{\frac{s_0}{2}}^{s_0}h'(s)^2 \,ds=\\
        P^2(\Omega)\int_{0}^{s_0}h'(s)^2 \,ds- c_n P^2(\Omega)\frac{\abs{\Omega}}{2P^\frac{n}{n-1}(\Omega)}\int_{0}^{\frac{s_0}{2}}h'(s)^2 \,ds,
    \end{aligned}
\end{equation}
hence \eqref{lambda:change} and \eqref{gradient:term} give

\begin{equation*}
    \lambda_\beta(\Omega)\le \dfrac{\int_{0}^{s_0}\left|h'(s)\right|^{2}ds+\beta h\left(s_0\right)^{2}}{\int_{0}^{s_0}\left|h(s)\right|^{2}ds}- c_n \frac{\abs{\Omega}}{2P^\frac{n}{n-1}(\Omega)}\dfrac{\int_{0}^{\frac{s_0}{2}}h'(s)^2 \,ds}{\int_{0}^{s_0}h^2(s)\, ds}.
\end{equation*}
Choosing $h(s)= cos(\sqrt{\nu_{1}}s)$, we get 
\begin{equation}\label{eq:const}
\begin{aligned}
   & \dfrac{\int_{0}^{\frac{\abs{\Omega}}{P(\Omega)}}\left|h'(s)\right|^{2}ds+\beta h\left(\frac{\abs{\Omega}}{P(\Omega)}\right)^{2}}{\int_{0}^{\frac{\abs{\Omega}}{P(\Omega)}}\left|h(s)\right|^{2}ds}=\nu_{1}(\beta, s_0)\\
    &\int_0^{s_0}h(s)^2\,ds=\frac{\abs{\Omega}}{2P(\Omega)}+\frac{\sin\left(2\sqrt{\nu_{1}}\frac{\abs{\Omega}}{P(\Omega)}\right)}{4\sqrt{\nu_{1}}}\leq \dfrac{\abs{\Omega}}{P(\Omega)},\\
    & \int_0^{\frac{s_0}{2}}h'(s)^2=\nu_{1}(\beta,s_0) \int_0^{\frac{s_0}{2}} \sin^2(\sqrt{\mu}s)\, ds\ge \frac{4}{\pi^2}\mu^2_1(\beta, s_0)\frac{s_0^3}{24}
    \end{aligned}
\end{equation}
Then equation \eqref{eq:const} and $P(\Omega)r(\Omega) \leq \abs{\Omega} n$, gives 
\begin{equation}
    \label{quasidef}
    \nu_{1}(\beta, s_0)-\lambda_\beta(\Omega)\ge \frac{c_n }{12n\pi^2} \frac{r(\Omega)}{P^{\frac{1}{n-1}}(\Omega)}\nu_{1}^2(\beta,s_0)s_0^2
\end{equation}
Again, combining \eqref{quasidef} with \eqref{eq:lowboundinradius}, \eqref{eq:perdiam}, and \eqref{bound:mu:s0}, we get
$$ \nu_{1}(\beta, s_0)-\lambda_\beta(\Omega)\ge \frac{c_n }{12n\pi^2} \frac{r(\Omega)}{P^{\frac{1}{n-1}}(\Omega)}\nu_{1}^2(\beta,s_0)s_0^2$$
 and so multiplying by $s_0^2$ and recalling \eqref{bound:mu:s0}, we get
$$ \frac{\pi^2}{4}\frac{1}{1+\frac{2P(\Omega)}{\beta \abs{\Omega}}}-\frac{\lambda_\beta(\Omega)\abs{\Omega}^2}{P^2(\Omega)}\ge C_5(n) \left(\frac{\pi^2}{4}\frac{1}{1+\frac{\pi^2\abs{\Omega}}{4\beta P(\Omega)}}\right)^2\mathcal{A}(\Omega).$$
The thesis follows using the upper bound in \eqref{convex_estimates}.
\end{proof}

\section*{Acknowledgements}
We would like to thank Francesco Della Pietra for the valuable advices. This work has been partially supported by GNAMPA group of INdAM. R. Sannipoli was supported by the grant no. 26-21940S
of the Czech Science Foundation.

\section*{Conflicts of interest and data availability statement}
The authors declare that there is no conflict of interest. Data sharing not applicable to this article as no datasets were generated or analyzed during the current study.

\Addresses
\bibliographystyle{plain}
\bibliography{biblio}
\end{document}